\documentclass[10pt]{amsart}
\usepackage{amssymb}
\usepackage{tikz-cd}

\title{on uniform and coarse rigidity  of $L^p([0,1])$}

\author {Christian Rosendal}
\address{Department of Mathematics\\University of Maryland\\4176 Campus Drive - William E. Kirwan Hall\\College Park, MD 20742-4015\\USA}
\email{rosendal@umd.edu}
\urladdr{https://sites.google.com/view/christian-rosendal/}

\newcommand{\norm}[1]{\lVert#1\rVert}
\newcommand{\Norm}[1]{\big\lVert#1\big\rVert}

\newcommand{\forkindep}[1][]{\mathop{\mathop{\vcenter{\hbox{\oalign{\noalign{\kern-.3ex}
\hfil$\vert$\hfil\cr\noalign{\kern-.7ex}$\smile$\cr\noalign{\kern-.3ex}}}}}\displaylimits_{#1}}}

\newcommand{\maths}[1]{\[\begin{split}{#1}\end{split}\]}

\newcommand {\R}{\mathbb R}
\newcommand {\C}{\mathbb C}

\newcommand{\saa}{\Rightarrow}
\newcommand{\equi}{\Leftrightarrow}

\newcommand{\Lim}[1]{\mathop{\longrightarrow}\limits_{#1}}

\newcommand {\del}{ \; \big| \;}

\newcommand{\inv}{^{-1}}

\theoremstyle{plain}
\newtheorem{thm}{Theorem}
\newtheorem{cor}[thm]{Corollary}

\newtheorem{quest}[thm]{Question}

\theoremstyle{definition}

\definecolor{groen}{rgb}{0,0.5,.7}
\definecolor{gul}{rgb}{0.94,0.8,0}
\definecolor{blaa}{rgb}{0.16,0,0.6}
\definecolor{roed}{rgb}{1,0,0}

\makeindex

\begin{document}

\thanks{The research was partially supported by the NSF through the award DMS 2204849. The author is also very thankful for the detailed criticisms from Bruno Braga that greatly improved the presentation in the paper.}

\maketitle

\begin{abstract}
If $X$ is an almost transitive Banach space with amenable isometry group (for example, if $X=L^p([0,1])$ with $1\leqslant p<\infty$) and $X$ admits a uniformly continuous map $X\overset\phi\longrightarrow E$ into a Banach space $E$ satisfying
$$
\inf_{\norm{x-y}=r}\Norm{\phi(x)-\phi(y)}>0
$$
for some $r>0$ (that is, $\phi$ is {\em almost uncollapsed}), then $X$ admits a simultaneously uniform and coarse embedding into a Banach space $V$ that is finitely representable in $L^2(E)$. 
\end{abstract}


\

The aim of the present paper is to prove a rigidity result regarding a priori very weak notions of embeddings between Banach spaces assuming that the domain space satisfies additional analytical assumptions. The general motivating problem for our study is the following still unresolved question\footnote{The origin of the question is not quite clear, but the lacking understanding of the relationship between uniform and coarse embeddings of Banach spaces was already pointed out in N. J. Kalton's survey paper \cite{kalton2}.}.
\begin{quest}\label{quest}
Are the following two properties equivalent for all Banach spaces $X$ and $E$?
\begin{enumerate}
\item[(a)] $X$ uniformly embeds into $E$,
\item[(b)] $X$ coarsely embeds into $E$.
\end{enumerate}
\end{quest}
Let us recall that a map $X\overset\phi\longrightarrow E$ is a {\em uniform embedding} if, for all pairs of sequences $x_n,y_n\in X$, we have 
$$
\lim_n\norm{x_n-y_n}=0\;\;\equi\;\; \lim_n\norm{\phi(x_n)-\phi(y_n)}=0,
$$
whereas it is a {\em coarse embedding} if, for all $x_n,y_n\in X$, we have 
$$
\lim_n\norm{x_n-y_n}=\infty\;\;\equi\;\; \lim_n\norm{\phi(x_n)-\phi(y_n)}=\infty.
$$
As noted, Question \ref{quest} remains open in general,  but the strongest result to date is that (a) implies (b) provided that $E\oplus E$ is isomorphic to a closed subspace of $E$ \cite{Ros2} (see also \cite{Ros1,Braga1} for related results). In fact, in \cite{Ros2} a stronger result is obtained, which states that, with the same assumption on $E$, if $X\overset\phi\longrightarrow E$ is a uniformly continuous map satisfying just that
\begin{equation}\label{uncollapsed}
\inf_{\norm{x-y}\geqslant r}\Norm{\phi(x)-\phi(y)}>0
\end{equation}
for some $r> 0$, then $X$ admits a simultaneously coarse and uniform embedding into $E$. Maps satisfying (1) for some $r> 0$ are called {\em uncollapsed} in \cite{Ros1} and this paper also introduces  the even weaker property that 
\begin{equation}\label{almost uncollapsed}
\inf_{\norm{x-y}=r}\Norm{\phi(x)-\phi(y)}>0
\end{equation}
for some $r> 0$, which, in turn,  is termed {\em almost uncollapsed} in \cite{Braga2}. Observe that both coarse and uniform embeddings are uncollapsed and a fortiori almost uncollapsed. Note also that, for example, the exponential map 
$$
t\in \R\mapsto e^{it}\in \C
$$
is almost uncollapsed, but not uncollapsed.
Furthermore, building on work of Kalton \cite{kalton}, Corollary 11 of \cite{Ros1} states that there is no uniformly continuous almost uncollapsed map from $c_0$ into a reflexive Banach space. In \cite{Braga2}, B. Braga takes these issues  further and shows that, if $X\overset\phi\longrightarrow E$ is a uniformly continuous almost uncollapsed map and $E$ has nontrivial type, then $q_X\leqslant q_E$, where 
$$
q_X=\inf\{q\in [2,\infty[\del X \text{ has cotype }q\}
$$
and similarly for $E$ (see Theorem 1.3 \cite{Braga2}). Braga's theorem is a variation of earlier breakthrough results by M. Mendel and A. Naor (Theorem 1.9 and Theorem 1.11 \cite{mendel}) giving the same conclusion provided that $X$ embeds either uniformly or coarsely in $E$.

Our main result, Theorem \ref{intro:main}, instead aims more directly at Question \ref{quest} by producing a simultaneously uniform and coarse embedding of $X$ into a new space with local properties similar to those of $E$. For example, note that cotype is preserved under both finite representability and under the passage from $E$ to $L^2(E)$. However, this comes at the cost of imposing a significant restriction on $X$ that is satisfied, for example, by $X=L^p([0,1])$, $1\leqslant p<\infty$, and $G={\sf Isom}(X)$.

\begin{thm}\label{intro:main}
Suppose $X$ is a Banach space admitting  a strongly continuous linear isometric action $G\curvearrowright X$ by an amenable topological group  so that $G$ induces a dense orbit on the unit sphere ${\sf S}_X$. Assume also there is a uniformly continuous map $X\overset\phi\longrightarrow E$ into a Banach space $E$ so that 
$$
\inf_{\norm{x-y}=r}\Norm{\phi(x)-\phi(y)}>0
$$
for some $r>0$.
Then $X$ admits a simultaneously uniform and coarse embedding into a Banach space $V$ that is finitely representable in $L^2(E)$.
\end{thm}

\begin{proof}Without loss of generality, we may assume that $\dim X\geqslant 2$.
By Proposition 7 of \cite{Ros1}, from $\phi$ we can construct another uniformly continuous map $X\overset\psi\longrightarrow  \ell^2(E)$, which, moreover, is {\em solvent}, i.e., so that, for some sequence $R_1, R_2, \ldots$ of constants and all $x,y\in X$, we have 
\begin{equation}\label{solvency}
R_n\leqslant \norm{x-y}\leqslant R_n+n\;\;\saa \;\;\norm{\psi(x)-\psi(y)}\geqslant n.
\end{equation}

Consider now the semidirect product group $X\rtimes G$ arising from the linear action of $G$ on $X$. For clarity of notation, elements $x\in X$ are denoted by $\tau_x$ and the group operation is written multiplicatively.  Thus every element of $X\rtimes G$ can be written uniquely in the form $\tau_x g$ for some $x\in X$ and $g\in G$. Moreover, for all $x,y\in X$ and $g\in G$, we have
$$
\tau_x\tau_y=\tau_{x+y} \quad\text{ and }\quad   g\tau_x =\tau_{g(x)}g.
$$
Define a pseudometric  $d$ on $X\rtimes G$ by $d(\tau_xg,\tau_yf)=\norm{x-y}$ and note that $d$ is left-invariant. Indeed, for all $x,y,z\in X$ and $g,f,h\in G$, we have
\[\begin{split}
d(\tau_zh\cdot \tau_xg, \tau_zh\cdot\tau_yf)
&=d(\tau_{z+h(x)}hg, \tau_{z+h(y)}hf)\\
&=\Norm{(z+h(x))-(z+h(y))}\\
&=\norm{h(x)-h(y)}\\
&=\norm{x-y}\\
&=d(\tau_xg, \tau_yf),
\end{split}\]
showing left-invariance. Note also that, if $X\rtimes G$ is given the topology induced by the identification with the cartesian product $X\times G$ via $(x,g)\mapsto \tau_xg$, then $d$ is continuous. In other words, $d$ is a continuous left-invariant pseudometric on the topological semidirect product $X\rtimes G$.

Define now a map $X\rtimes G\overset\Psi\longrightarrow \ell^2(E)$ by setting $\Psi(\tau_xg)=\psi(x)$ and note that $\Psi$ is uniformly continuous with respect to the pseudometric $d$. For $t>0$, let also
$$
\eta(t)=\inf_{d(\tau_xg,\tau_yf)=t}\norm{\Psi(\tau_xg)-\Psi(\tau_yf)}=\inf_{\norm{x-y}=t}\norm{\psi(x)-\psi(y)}
$$
and 
$$
\mu(t)=\sup_{d(\tau_xg,\tau_yf)\leqslant t}\norm{\Psi(\tau_xg)-\Psi(\tau_yf)}=\sup_{\norm{x-y}\leqslant t}\norm{\psi(x)-\psi(y)}.
$$

As $X$ is abelian and $G$ amenable, also $X\rtimes G$ is amenable (see, for example,  Theorem G.2.1 and Proposition G.2.2.(ii) in \cite{bekka}). So, by Theorem 6.1 of \cite{thom}, which improves Theorem 16 of \cite{Ros1}, there is a Banach space $V$, that is finitely representable in $L^2(\ell^2(E))$ and thus also in $L^2(E)$, and a continuous linear isometric action $X\rtimes G\overset\pi\curvearrowright V$ with an associated continuous cocycle $X\rtimes G\overset b\longrightarrow V$ satisfying the bounds
\begin{equation}\label{one}
\eta\big(d(\tau_xg,\tau_yf)\big)
\leqslant \Norm{b(\tau_xg)-b( \tau_yf)}\leqslant \mu\big(d(\tau_xg,\tau_yf)\big),
\end{equation}
that is,
\maths{
\eta\big(\norm{x-y}\big)
&\leqslant \Norm{b(\tau_xg)-b( \tau_yf)}\leqslant \mu\big(\norm{x-y}\big)
}
for all $x,y\in X$ and $g,f\in G$. Here, that $b$ is a cocycle simply means that it satisfies the equation
\begin{equation}\label{two}
b(\sigma\gamma)=\pi(\sigma)b(\gamma)+b(\sigma)
\end{equation}
for all $\sigma, \gamma\in X\rtimes G$.

Observe that, by the solvency condition  (\ref{solvency}) on $\psi$, we have 
$\eta(t)\geqslant n$ whenever $R_n\leqslant t\leqslant R_n+n$. It thus follows that
$$
R_n\leqslant \norm{x-y}\leqslant R_n+n\;\saa\;  \Norm{b(\tau_xg)-b( \tau_yf)}\geqslant n
$$
for all $n\geqslant 1$, $x,y\in X$ and $g,f\in G$.

Note also that, by the cocycle equation (\ref{two}), we have for $\sigma, \gamma\in X\rtimes G$ that 
$$
b(\gamma)=b(\sigma\sigma\inv\gamma)=\pi(\sigma)b(\sigma\inv \gamma)+b(\sigma),
$$
whereby
$$
\norm{b(\gamma)-b(\sigma)}=\norm{\pi(\sigma)b(\sigma\inv \gamma)}=\norm{b(\sigma\inv \gamma)}.
$$
Observe also that, because $\tau_0=1=\tau_01$ is the identity in $X\rtimes G$, we have $b(\tau_01)=0$ and therefore, for all $g\in G$, 
$$
\norm{b(g)}=\norm{b(\tau_0g)-b( \tau_01)}\leqslant \mu\big(\norm{0-0}\big)=0,
$$
that is $b(g)=0$.

Suppose now that $x$ and $y$ are two elements of $X$ with $\norm{x}=\norm y$. Then, because some and hence every $G$-orbit on ${\sf S}_X$ is dense, there are $g_n\in G$ so that $g_n(x)\Lim{n\to \infty}y$. It thus follows from continuity of $b$ and the cocycle equation that
\[\begin{split}
\Norm{b(\tau_y)}
&=\lim_{n\to \infty}\Norm{b(\tau_{g_n(x)})}\\
&=\lim_{n\to \infty}\Norm{b(g_n\tau_{x}g_n\inv)}\\
&=\lim_{n\to \infty}\Norm{\pi(g_n\tau_{x})b(g_n\inv)+\pi(g_n)b(\tau_x)+b(g_n)}\\
&=\lim_{n\to \infty}\Norm{\pi(g_n\tau_{x})0+\pi(g_n)b(\tau_x)+0}\\
&=\lim_{n\to \infty}\Norm{b(\tau_x)}\\
&=\Norm{b(\tau_x)}.
\end{split}\]
In other words, $\Norm{b(\tau_x)}=\Norm{b(\tau_y)}$ whenever $\norm x=\norm y$.

Let $x\in X$ be any element with $\alpha=\norm x\geqslant R_n$. We claim that $\norm{b(\tau_x)}\geqslant \frac n2$. To see this, pick any $y\in X$ with $\norm y=R_n$. Because $\dim X\geqslant 2$,  $y+\alpha {\sf S}_X$ is a path connected set containing both points of norm $\leqslant \alpha$ and points of norm $\geqslant \alpha$.  Therefore, $y+\alpha {\sf S}_X$ must intersect $\alpha {\sf S}_X$ and thus $y=z+u$ for some $z,u$ with $\norm z=\norm u=\alpha=\norm x$. Because $R_n\leqslant \norm{y-0}\leqslant R_n+n$, it follows that 
\[\begin{split}
n
&\leqslant \norm{b(\tau_y)-b(\tau_0)}\\
&=\norm{b(\tau_y)}\\
&=\norm{b(\tau_z\tau_u)}\\
&=\norm{\pi(\tau_z)b(\tau_u)+b(\tau_z)}\\
&\leqslant\norm{\pi(\tau_z)b(\tau_u)}+\norm{b(\tau_z)}\\
&=\norm{b(\tau_u)}+\norm{b(\tau_z)}\\
&=2\norm{b(\tau_x)},
\end{split}\]
and so $\norm{b(\tau_x)}\geqslant \frac n2$, as claimed.

From this it follows that, for all $n\geqslant 1$ and $x,y\in X$,
$$
\norm{x-y}\geqslant R_n\;\saa\;  \norm{b(\tau_x)-b(\tau_y)}=\norm{b(\tau_{y}\inv\tau_x)}=\norm{b(\tau_{x-y})}\geqslant \frac n2.
$$
Thus, if we restrict $X\rtimes G\overset b\longrightarrow V$ to the factor $X$, we obtain a continuous cocycle $X\overset b\longrightarrow V$ for the action $X\overset \pi\curvearrowright V$ satisfying also
$$
\lim_n\norm{x_n-y_n}=\infty\;\;\saa\;\; \lim_n\Norm{b(x_n)-b(y_n)}=\infty.
$$
Applying Proposition 1 \cite{Ros1}, we see that $X\overset b\longrightarrow V$ is uniformly continuous and that there are constants $c,C>0$ so that, for all $x,y\in X$,
$$
c \cdot \min\{   \norm{x-y}, 1  \}
\leqslant  
\norm{b(x)-b(y)}
\leqslant C\norm{x-y}+C.
$$
From this it follows that $b$ is simultaneously a uniform and coarse embedding of $X$ into $V$.  
\end{proof}

For $1\leqslant p<\infty$, the Banach space $L^p([0,1])$ is {\em almost transitive} (Theorem 9.6.3 and Theorem 9.6.4 \cite{rolewicz}), that is, has a dense orbit on the unit sphere under the action of its linear isometry group ${\sf Isom}\big(L^p([0,1])\big)$. Furthermore, ${\sf Isom}\big(L^p([0,1])\big)$ is amenable when equipped with the strong operator topology (in fact, by Theorem 6.6 of \cite{pestov}, it is extremely amenable).  The following corollary is therefore immediate.

\begin{cor}
Let $1\leqslant p<\infty$ and suppose that there is a  uniformly continuous map $L^p([0,1])\overset\phi\longrightarrow E$ into a Banach space so that 
$$
\inf_{\norm{x-y}=r}\Norm{\phi(x)-\phi(y)}>0
$$
for some $r>0$.
Then $L^p([0,1])$ admits a simultaneously uniform and coarse embedding into a Banach space $V$ that is finitely representable in $L^2(E)$.
\end{cor}
This result also supports the view that uniform continuity is the most significant requirement with regards to uniform embeddings, whereas uniform continuity of the inverse is easier to obtain. In this connection, A. Naor \cite{Naor} gives an example of a bornologous map between two separable Banach spaces that is not close to any uniformly continuous map\footnote{A map $M\overset\phi\longrightarrow N$ between two metric spaces is said to be {\em bornologous} provided that $\sup_{d(x,y)\leqslant t}d(\phi(x),\phi(y))<\infty$ for all $t<\infty$. Moreover, two maps $M\overset{\phi,\psi}\longrightarrow N$ are {\em close} provided that $\sup_xd(\phi(x),\psi(x))<\infty$.}.


\begin{thebibliography}{99}

\bibitem{bekka}Bachir Bekka, Pierre de la Harpe and Alain  Valette, {\em Kazhdan's property (T)}. New Mathematical Monographs, 11. Cambridge University Press, Cambridge, 2008. 

\bibitem{Braga1}Bruno de Mendon\c{c}a Braga, {\em Coarse and uniform embeddings}, Journal of Functional Analysis 272 (2017), no. 5, 1852--1875.

\bibitem{Braga2}Bruno de Mendon\c{c}a Braga, {\em On weaker notions of  nonlinear embeddings between Banach spaces}, Journal of Functional Analysis, 274 (2018), no. 11, 3149--3169.


\bibitem{pestov}Thierry Giordano and Vladimir Pestov, {\em Some extremely amenable groups related to operator algebras and ergodic theory}, Journal of the Institute of Mathematics of Jussieu, 6(2) (2007), 279--315.

\bibitem{kalton}Nigel J. Kalton, {\em Coarse and uniform embeddings into reflexive spaces}, Quart. J. Math. 58(3)
(2007), 393--414.

\bibitem{kalton2}Nigel J. Kalton, {\em The nonlinear geometry of Banach spaces}, Rev. Mat. Complut. 21 (2008), no. 1, 7--60. 


\bibitem{mendel}Manor Mendel and Assaf Naor, {\em Metric Cotype}, Annals of Mathematics, 168 (2008), 247--298.

\bibitem{Naor}Assaf Naor, {\em Uniform nonextendability from nets}, C. R. Acad. Sci.- Series I - Math\'ematique 
353(11) (2015), 991--994.



\bibitem{rolewicz}Stefan Rolewicz, {\em Metric linear spaces}, Second edition. Mathematics and its Applications (East European Series), 20. D. Reidel Publishing Co., Dordrecht; PWN--Polish Scientific Publishers, Warsaw, 1985.



\bibitem{Ros1} Christian Rosendal, {\em Equivariant geometry of Banach spaces and topological groups}, Forum of Mathematics, Sigma (2017), Vol. 5, e22, 62 pages.


\bibitem{Ros2}Christian Rosendal, {\em Geometries of topological groups}, preprint. 

\bibitem{thom}Friedrich Martin Schneider and Andreas Thom, {\em On F\o lner sets in topological groups}, Compositio Mathematica, Volume 154, Issue 7, July 2018, pp. 1333--1361.

\end{thebibliography}
\end{document}